\documentclass{proc-l}

\usepackage{amsmath,amsfonts,amsthm}

\usepackage[numbers]{natbib} 
\usepackage{hyperref} 
\usepackage{xcolor} 
\usepackage{dsfont} 
\usepackage{enumerate} 
\usepackage{subdepth} 

\newcommand{\C}{\mathbb{C}}
\newcommand{\N}{\mathbb{N}}
\newcommand{\R}{\mathbb{R}}
\newcommand{\EE}{\mathsf{E}} 
\newcommand{\bb}[1]{\boldsymbol{#1}}
\newcommand{\rd}{\hskip1pt\mathrm{d}\hskip0.25pt}
\newcommand{\etr}{\mathop{\mathrm{etr}}\hskip1pt}

\newtheorem{theorem}{Theorem}[section]
\newtheorem{lemma}[theorem]{Lemma}
\newtheorem{corollary}[theorem]{Corollary}
\newtheorem{remark}[theorem]{Remark}
\newtheorem{conjecture}[theorem]{Conjecture}

\begin{document}

\title[Wishart moment for the product of two disjoint principal minors]{An explicit Wishart moment formula for the product of two disjoint principal minors}

\author[C.\ Genest]{Christian Genest}
\address{Department of Mathematics and Statistics, McGill University, Montr\'eal (Qu\'ebec) Canada H3A 0B9}
\email{christian.genest@mcgill.ca}
\thanks{Genest's research is funded in part by the Canada Research Chairs Program and the Natural Sciences and Engineering Research Council of Canada.}

\author[F.\ Ouimet]{Fr\'ed\'eric Ouimet}
\address{Department of Mathematics and Statistics, McGill University, Montr\'eal (Qu\'ebec) Canada H3A 0B9}
\email{frederic.ouimet2@mcgill.ca}
\thanks{Ouimet's funding was made possible through a contribution to Christian Genest's research program from the Trottier Institute for Science and Public Policy.}

\author[D.\ Richards]{Donald Richards}
\address{Department of Statistics, Penn State University, University Park, PA 16802, USA}
\email{richards@stat.psu.edu}


\date{***}

\commby{}

\begin{abstract}
This paper provides the first explicit formula for the expectation of the product of two disjoint principal minors of a Wishart random matrix, solving a part of a broader problem put forth by Samuel S.\ Wilks in 1934 in the \textit{Annals of Mathematics}. The proof makes crucial use of hypergeometric functions of matrix argument and their Laplace transforms. Additionally, a Wishart generalization of the Gaussian product inequality conjecture is formulated and a stronger quantitative version is proved to hold in the case of two minors.
\end{abstract}

\maketitle

\section{Introduction}\label{sec:intro}

The Wishart distribution was first derived in 1928 by the Scottish statistician John Wishart while studying the joint distribution of empirical covariances in multivariate normal populations; see \cite{doi:10.2307/2331939}. In its full generality, the Wishart distribution is a matrix-variate extension of the Gamma distribution which has been used in statistical analysis for decades to model linear dependence between multiple variables in various scientific disciplines and research areas. Its application spans from statistical inference, facilitating hypothesis testing and the construction of confidence regions for covariance matrices, to Bayesian statistics, where it shines as a conjugate prior for precision matrices of multivariate normal distributions.
It also proves useful in optimizing investment portfolios in finance through the estimation of the variability of asset returns and their dependence structure. This broad use underscores the Wishart distribution's great importance in multivariate analysis, providing deep insights into both theoretical and applied problems.

Let $\mathcal{S}_{++}^p$ denote the set of $p\times p$ real (symmetric) positive definite matrices. For arbitrary degree-of-freedom parameter $\alpha\in (p-1,\infty)$ and scale matrix $\Sigma\in \mathcal{S}_{++}^p$, the corresponding Wishart probability density function is defined for every $X\in \mathcal{S}_{++}^p$, relative to the Lebesgue measure $\rd X$ on $\mathcal{S}_{++}^p$, by
\[
f_{\alpha,\Sigma}(X) = \frac{|X|^{\alpha/2 - (p + 1)/2} \etr(-\Sigma^{-1} X/2)}{|2 \Sigma|^{\alpha/2} \Gamma_p(\alpha/2)},
\]
where $\mathrm{etr}(\cdot) \equiv \exp\{\mathrm{tr}(\cdot)\}$ stands for the exponential of the trace, $|\cdot|$ denotes the determinant, and the multivariate gamma function $\Gamma_p(\cdot)$ is defined in \eqref{eq:Gamma.m.int} below. If a random matrix $\mathfrak{X}$ follows this distribution, one writes $\mathfrak{X}\sim \mathcal{W}_p(\alpha,\Sigma)$.

A simple renormalization of the density $f_{\alpha,\Sigma}$ shows that the moment-generating function of $\mathfrak{X}\sim \smash{\mathcal{W}_p(\alpha,\Sigma)}$ is given, for every real symmetric matrix $T$ of size $p\times p$ such that $\Sigma^{-1} - 2T\in \mathcal{S}_{++}^p$, by
\begin{equation}\label{eq:Lt}
\EE\{\etr(T \mathfrak{X})\} = |I_p - 2 T \Sigma|^{-\alpha/2};
\end{equation}
see, e.g., Theorem~3.2.3 of \citet{MR652932}.

Despite the moment-generating function of the Wishart distribution being widely known, the problem of calculating the joint moments of disjoint principal minors of Wishart random matrices has remained unsolved for the past 90 years, ever since Samuel S.\ Wilks raised the issue in 1934 in the \textit{Annals of Mathematics}; see \citet[pp.~325--326]{MR1503165}. More specifically, if the random matrix $\mathfrak{X}$ is expressed in the form $(\mathfrak{X}_{ij})$, where for each $i,j\in \{1,\ldots,d\}$, the block $\mathfrak{X}_{ij}$ has size~$p_i \times p_j$, so that $p_1 + \dots + p_d = p$, then the problem consists of calculating, for every $\nu_1,\ldots,\nu_d\in [0,\infty)$, the expectation
\begin{equation}\label{eq:Wilks.problem}
\EE\!\left(\prod_{i=1}^d |\mathfrak{X}_{ii}|^{\nu_i}\right).
\end{equation}
This question, which Wilks described as ``extremely complicated,'' arose from his intensive research into the moment properties of various generalized statistics pertaining to multivariate linear regression and the analysis of variance; see \cite{doi:10.2307/2331979,MR1503165}.

The first result of this paper (Theorem~\ref{thm:1}) solves Wilks' problem completely in the case of two minors ($d = 2$) by providing the first explicit formula for the expectation in \eqref{eq:Wilks.problem} since Wilks posed the problem in 1934. The derived formula involves the Gaussian hypergeometric function of matrix argument and holds true for a range of exponents wider than $[0,\infty)$.

Surprisingly, even in the much simpler setting of a centered multivariate normal random vector $\mathbf{Z} = (Z_1,\ldots,Z_d)$ with covariance matrix $\Sigma$, a general formula for the joint absolute moments
\begin{equation}\label{eq:Gaussian.moments}
\EE\!\left( \prod_{i=1}^d |Z_i|^{2 \nu_i} \right),
\end{equation}
expressed using known special functions, has only been solved in the bivariate case by \citet{MR45347}; the trivariate case is still open \citep{MR52072}. The most recent development in this direction is due to \citet{MR4219331}, where series representations are established for \eqref{eq:Gaussian.moments} in all dimensions, including in the noncentral case and under sectional truncations. Nevertheless, there is currently no unified framework in terms of special functions beyond two dimensions; see Corollary~3 of \cite{MR4219331}.

Other types of Wishart moments have been studied in the real case; see, e.g., \citet{MR181056} and \citet{MR4757806} for the joint moments of embedded principal minors of a Wishart random matrix, \citet{MR1868979} for the joint moments of monomials in the entries of a Wishart random matrix, \citet{MR2066255} for the expectations of polynomials of Wishart (and inverse Wishart) random matrices which are invariant under the conjugate action of the orthogonal group, \citet{MR2458187} for the first- and second-order moments of any two (i.e., not necessarily principal) minors of a Wishart random matrix, and \citet{doi:10.1007/s40953-021-00267-7,doi:10.1111/sjos.12707} for the expectations of matrix-valued functions of Wishart (and inverse Wishart) random matrices which are equivariant under the conjugate action of the orthogonal group.

On the one hand, the quest to obtain a general explicit formula for the expectation in \eqref{eq:Wilks.problem} is motivated by its many potential applications to statistical problems including: multivariate linear regression, the analysis of variance, calculation of the moments of generalized correlation coefficients, and the probability distributions of generalized Student's~$t$ ratios. On the other hand, the desire for a general explicit formula for the expectation in \eqref{eq:Wilks.problem} is fueled, in part, by its potential use in addressing the Wishart analog of the Gaussian product inequality (GPI) conjecture introduced by~\citet{arXiv:2311.00202}. The conjecture is stated below for convenience, along with a list of possible implications, raised as open questions.

\begin{conjecture}\label{conj:1}
Let $\mathfrak{X}\sim \mathcal{W}_p(\alpha,\Sigma)$ for given $p\in \N$, $\alpha\in (p-1,\infty)$ and $\Sigma\in \mathcal{S}_{++}^p$. Then, for every $\nu_1,\ldots,\nu_d\in [0,\infty)$,
\[
\EE\!\left(\prod_{i=1}^d |\mathfrak{X}_{ii}|^{\nu_i}\right) \geq \prod_{i=1}^d \EE(|\mathfrak{X}_{ii}|^{\nu_i}).
\]
\end{conjecture}

\begin{enumerate}[\quad1.]
\item Can the conjecture help in deriving explicit bounds for the power function of the likelihood ratio test statistic \citep[Chapter~11]{MR652932} for testing independence among $d$ sets of variables with a joint multivariate normal distribution?
\item Given that Gaussian moments are connected to algebraic concepts such as permanents and Hafnians through Wick's formula, is a similar relationship true for Wishart moments? Could the conjecture help in discovering new algebraic inequalities? Notably, for $1\times 1$ determinants (i.e., the multivariate gamma distribution), moments are associated with $\alpha$-permanents \citep{MR1450811}.
\item Is there a Wishart analog to the real linear polarization constant conjecture \citep{MR1636556,MR3425898} in functional analysis?
\item The proof of the Gaussian correlation inequality \citep{MR3289621} is more straightforward in the framework of the multivariate gamma distribution. Could the GPI also become simpler when considered in the context of the Wishart or multivariate gamma distributions?
\end{enumerate}

The second result of this paper (Corollary~\ref{cor:1}) shows that a stronger quantitative version of Conjecture~\ref{conj:1} holds true when there are two minors ($d = 2$).

The GPI conjecture itself states that, for every $\nu_1, \ldots, \nu_d\in [0,\infty)$,
\begin{equation}\label{eq:GPI}
\mathrm{GPI}_d(\bb{\nu}) : \quad \EE\!\left( \prod_{i=1}^d |Z_i|^{2 \nu_i} \right) \geq \prod_{i=1}^d \EE(|Z_i|^{2 \nu_i}).
\end{equation}
In the special case $\nu_1 = \dots = \nu_d\in \N_0 = \{0, 1, \ldots\}$, this inequality corresponds to the real linear polarization constant problem in functional analysis, initially posed by \citet{MR1636556} and later recontextualized in the above Gaussian framework by \citet{MR2385646}. The complex version of the linear polarization constant problem, being simpler, was established by \citet{MR1653495}.

The GPI has been extensively studied, yielding several significant partial results. It is known that $\mathrm{GPI}_2(\nu_1, \nu_2)$ holds for all $(\nu_1, \nu_2) \in [0,\infty)^2$ due to the MTP$_2$ property of $(|Z_1|, |Z_2|)$, as shown by \citet{MR628759}. Furthermore, $\mathrm{GPI}_d(1, \dots, 1)$ is satisfied for all $d \in \mathbb{N}$, as established by \citet{MR2385646}, combining Wick's formula with properties of permanents and Hafnians.

\citet{MR4052574} proved through an analysis involving the Gaussian hypergeometric function that $\mathrm{GPI}_3(m, m, n)$ holds. \citet{MR4445681,MR4760098} validated the GPI for specific tuples in dimensions $d \in \{3, 4, 5\}$ using combinatorial and computer-assisted sums-of-squares approaches. Those results were extended subsequently by \citet{MR4661091}, who used mathematical induction, Lagrange multipliers, and Gaussian integration-by-parts to establish that $\mathrm{GPI}_3(\bb{n})$ holds for all $\bb{n}\in \smash{\N_0^3}$. Using a similar induction approach, \citet{MR4794515} extended \cite{MR4661091} to include positive half-integers.

The GPI has also been investigated under specific assumptions on the covariance matrix of $\bb{Z} = (Z_1,\ldots,Z_d)$, denoted $\Sigma$. \citet{MR4466643} showed that $\mathrm{GPI}_d(\bb{n})$ holds for any $d \in \mathbb{N}$ and $\bb{n} \in \smash{\mathbb{N}_0^d}$ when $\Sigma$ is completely positive, using a combinatorial method related to the log-convexity of the gamma function. \citet{MR4445681} proved that $\mathrm{GPI}_d(\bb{n})$ holds under the weaker assumption that $\Sigma$ is entry-wise nonnegative, employing an Isserlis–Wick type formula. \citet{MR4554766} extended the latter to the case where the $Z_i^2$'s are replaced by the components of a multivariate gamma random vector.

Another direction explored in the literature involves replacing the positive powers in \eqref{eq:GPI} with other functions of $Z_i^2$ (or functions of the components of more general random vectors) and showing that the inequality holds. The case of Hermite polynomials of Gaussian components was considered by \citet{MR3425898}. Also in the Gaussian setting, the case of integrals of cosine functions was treated by \citet{MR3608204}; see \cite{MR4538422} for several functional and distributional generalizations. The case of negative powers was solved entirely by \citet{MR3278931} in the multivariate gamma setting and was extended to the Wishart setting in \cite{MR4538422,arXiv:2311.00202}, where it was also shown that negative powers can be replaced more generally by completely monotone functions. The case of mixed exponents (positive and negative powers) was studied with an increasing level of generality in \cite{MR4471184}, \cite{MR4530374}, \cite{MR4666255}, and most recently in \cite{arXiv:2311.00202}. Refer to \cite{arXiv:2311.00202} for a more comprehensive summary of facts related to the GPI conjecture.

Here is a brief outline of the present paper. In Section~\ref{sec:definitions}, some definitions and notations are introduced. The two results are stated in Section~\ref{sec:main.results} and proved in Section~\ref{sec:proofs}. Numerical validation of the results is provided in Appendix~\ref{app:numerical.validation}. Two technical lemmas used in the proofs are relegated to Appendix~\ref{app:tech.lemmas}.

\section{Definitions}\label{sec:definitions}

For any positive integer $m \in \N = \{1,2,\ldots\}$, the multivariate gamma function~$\Gamma_m$ is defined, for any $\beta \in \C$ such that $\textrm{Re}(\beta) > (m-1)/2$, by
\begin{equation}\label{eq:Gamma.m.int}
\Gamma_m(\beta) = \int_{\mathcal{S}_{++}^m} \etr(- X) \, |X|^{\beta - (m + 1)/2} \rd X,
\end{equation}
which is a natural generalization to the cone $\mathcal{S}_{++}^m$ of the classical gamma function. It is well known \cite[Theorem~2.1.12]{MR652932} that this integral has the exact evaluation
\begin{equation}\label{eq:Gamma.m.prod}
\Gamma_m(\beta) = \pi^{m(m-1)/4} \prod_{j=1}^m \Gamma ( \beta - (j - 1)/2 ).
\end{equation}

For any nonnegative integer $k$, a \textit{partition} $\bb{\kappa} = (k_1,\ldots,k_m)$ of $k \equiv |\bb{\kappa}|$ is a vector of nonnegative integers such that $k_1 \geq \dots \geq k_m$ and $k_1 + \dots + k_m = k$. Following \citet[Definition~7.2.1]{MR652932}, there corresponds a \textit{zonal polynomial} $C_{\bb{\kappa}}(M)$ for each partition $\bb{\kappa}$ and complex symmetric matrix $M$ of size $m \times m$. For a comprehensive introduction to zonal polynomials, refer, e.g., to \cite{MR882715, MR744672}.

Define the classical rising factorial, $(a)_k = \prod_{i=0}^{k-1} (a + i)$, with $(a)_0 = 1$. Further, define the partitional rising factorial $(a)_{\bb{\kappa}}$ by
\begin{equation}\label{eq:partitional.rising}
(a)_{\bb{\kappa}} = \prod_{j=1}^m (a - (j-1)/2)_{k_j}.
\end{equation}
Then one follows \cite[Definition~7.3.1]{MR652932} to define the generalized hypergeometric function of matrix argument ${}_pF_q$, for all $p,q\in \N_0$ and $a_1,\ldots,a_p,b_1,\ldots,b_q\in \C$, by
\begin{equation}\label{eq:hgfoma}
{}_pF_q(a_1,\ldots,a_p;b_1,\ldots,b_q;M) = \sum_{k=0}^{\infty} \sum_{\bb{\kappa}} \frac{(a_1)_{\bb{\kappa}} \cdots (a_p)_{\bb{\kappa}}}{(b_1)_{\bb{\kappa}} \cdots (b_q)_{\bb{\kappa}}} \frac{C_{\bb{\kappa}}(M)}{k!},
\end{equation}
where $M$ is an $m\times m$ complex symmetric matrix, the inner sum $\sum_{\bb{\kappa}}$ denotes a summation over all partitions $\bb{\kappa} = (k_1,\ldots,k_m)$ of $k$, $C_{\bb{\kappa}}(M)$ is the zonal polynomial of $M$ corresponding to $\bb{\kappa}$, and the complex numbers $b_1,\ldots,b_q$ are such that $(b_j)_{\bb{\kappa}} \neq 0$ for all $\bb{\kappa}$ and all $j\in \{1,\ldots,q\}$. The convergence properties of the zonal polynomial series \eqref{eq:hgfoma} are provided in Theorem~6.3 of ~\citet{MR882715}. In particular, in the case of the Gaussian hypergeometric function of matrix argument, for which $(p,q)=(2,1)$, the series \eqref{eq:hgfoma} converges absolutely for all $m\times m$ complex symmetric matrices $M$ such that $\|M\| < 1$, where $\|\cdot\|$ denotes the spectral norm.

\section{Results}\label{sec:main.results}

The theorem below presents a general formula for the expectation of the product of two disjoint principal minors of a Wishart random matrix. A further generalization is also stated with the additional factor $\etr(T \mathfrak{X}) |\mathfrak{X}|^{\nu_0}$ in the expectation, which renormalizes the Wishart measure and offers more flexibility.

\begin{theorem}\label{thm:1}
For any $d, p_1, p_2 \in \N$, let $p = p_1 + p_2$ with $p_1 \leq p_2$, and fix any degree-of-freedom $\alpha \in (p-1, \infty)$. Let $\Sigma\in \mathcal{S}_{++}^p$ be any $p \times p$ positive definite matrix expressed in the form $(\Sigma_{ij})$, where for each $i,j\in \{1,2\}$, the block $\Sigma_{ij}$ has size~$p_i \times p_j$. Denote $A = \Sigma^{-1}/2$ and $\smash{P = A_{11}^{-1/2} A_{12} A_{22}^{-1/2}}$. Suppose that
\[
\mathfrak{X} =
\begin{bmatrix}
\mathfrak{X}_{11} & \mathfrak{X}_{12} \\
\mathfrak{X}_{21} & \mathfrak{X}_{22} \\
\end{bmatrix}
\sim \mathcal{W}_p(\alpha,\Sigma), \quad \text{with}~~~
\Sigma =
\begin{bmatrix}
\Sigma_{11} & \Sigma_{12} \\
\Sigma_{21} & \Sigma_{22} \\
\end{bmatrix}.
\]
Then, for every real $\nu_1 > -\alpha/2 + (p_1 - 1)/2$ and $\nu_2 > -\alpha/2 + (p_2 - 1)/2$, one has
\begin{equation}\label{thm:1.eq.1}
\begin{aligned}
\EE(|\mathfrak{X}_{11}|^{\nu_1} |\mathfrak{X}_{22}|^{\nu_2})
&= |2 \Sigma_{11}|^{\nu_1} \frac{\Gamma_{p_1}(\alpha/2 + \nu_1)}{\Gamma_{p_1}(\alpha/2)} \, |2 \Sigma_{22}|^{\nu_2} \frac{\Gamma_{p_2}(\alpha/2 + \nu_2)}{\Gamma_{p_2}(\alpha/2)} \\[-0.5mm]
&\quad\times {}_2F_1(-\nu_1, -\nu_2; \alpha/2; P P^{\top}) \\
&= \EE(|\mathfrak{X}_{11}|^{\nu_1}) \, \EE(|\mathfrak{X}_{22}|^{\nu_2}) \, {}_2F_1(-\nu_1, -\nu_2; \alpha/2; P P^{\top}).
\end{aligned}
\end{equation}
More generally, for any $p\times p$ real symmetric matrix $T$ such that $\Sigma^{-1} - 2T\in \mathcal{S}_{++}^p$, consider
\[
\mathfrak{X}_T \sim \mathcal{W}_p\{\alpha + 2\nu_0, (\Sigma^{-1} - 2T)^{-1}\}.
\]
Further, let $A_T = \Sigma^{-1}/2 - T$ and $P_T= (A_T)_{11}^{-1/2} (A_T)_{12} (A_T)_{22}^{-1/2}$. Then, for every $\nu_0 > -\alpha/2 + (p - 1)/2$, $\nu_1 > -\alpha/2 - \nu_0 + (p_1 - 1)/2$, $\nu_2 > -\alpha/2 - \nu_0 + (p_2 - 1)/2$, one has
\begin{equation}\label{thm:1.eq.2}
\begin{aligned}
&\EE\{\etr(T \mathfrak{X}) \, |\mathfrak{X}|^{\nu_0} |\mathfrak{X}_{11}|^{\nu_1} |\mathfrak{X}_{22}|^{\nu_2}\} \\
&\quad= \EE\{\etr(T \mathfrak{X})\} \, \EE(|\mathfrak{X}_T|^{-\nu_0})^{-1} \, \EE\{|(\mathfrak{X}_T)_{11}|^{\nu_1}\} \, \EE\{|(\mathfrak{X}_T)_{22}|^{\nu_2}\} \\
&\qquad\times {}_2F_1(-\nu_1, -\nu_2; \alpha_0/2; P_T P_T^{\top}).
\end{aligned}
\end{equation}
\end{theorem}

\begin{remark}
In Theorem~\ref{thm:1}, the ranges for $\nu_0,\nu_1,\nu_2$ are the largest possible that result in all expectations and integrals in the proof being well-defined and finite.
\end{remark}

\begin{remark}
For the case in which $p_1 > p_2$, one can easily obtain moment formulas analogous to those of Theorem~\ref{thm:1} by applying the transformation
\[
\mathfrak{X}(\Pi)
= \Pi \, \mathfrak{X} \, \Pi^{\top}
=
\begin{bmatrix}
\mathfrak{X}_{22} & \mathfrak{X}_{21} \\
\mathfrak{X}_{12} & \mathfrak{X}_{11}
\end{bmatrix},
\quad \text{with}~~~
\Pi =
\begin{bmatrix}
0_{p_2 \times p_1} & I_{p_2} \\
I_{p_1} & 0_{p_1 \times p_2}
\end{bmatrix}.
\]
This transformation leads to $\mathfrak{X}(\Pi) \sim \mathcal{W}_p\{\alpha,\Sigma(\Pi)\}$, where $\Sigma(\Pi) = \Pi \, \Sigma \, \Pi^{\top}$. As a result, if $A(\Pi)$ and $P(\Pi)$ correspond to the matrices $A$ and $P$ after the transformation, then $A(\Pi) = (\Pi \, \Sigma \, \Pi^{\top})^{-1} / 2 = \Pi \, A \, \Pi^{\top}$, and consequently, $\smash{P(\Pi) = A_{22}^{-1/2} A_{21} A_{11}^{-1/2}}$. Therefore, if $p_1 > p_2$, one finds
\[
\EE(|\mathfrak{X}_{11}|^{\nu_1} |\mathfrak{X}_{22}|^{\nu_2})
= \EE(|\mathfrak{X}_{11}|^{\nu_1}) \, \EE(|\mathfrak{X}_{22}|^{\nu_2}) \, {}_2F_1\{-\nu_1, -\nu_2; \alpha/2; P(\Pi) P(\Pi)^{\top}\}.
\]
The formula that would replace \eqref{thm:1.eq.2} is completely analogous.
\end{remark}

The corollary below establishes a stronger quantitative version of Conjecture~\ref{conj:1} when there are two minors ($d = 2$). The proof relies on Theorem~\ref{thm:1} and a certain lower bound on the Gaussian hypergeometric function of matrix argument ${}_2F_1$.

\begin{corollary}\label{cor:1}
Suppose that the assumptions of Theorem~\ref{thm:1} hold. Then, for every real symmetric matrix $T$ of size $p\times p$ such that $\Sigma^{-1} - 2T\in \mathcal{S}_{++}^p$, and every real $\nu_0 > -\alpha/2 + (p - 1)/2$, and $\nu_1,\nu_2\in [0,\infty)$, one has
\[
\begin{aligned}
&\EE\{\etr(T \mathfrak{X}) \, |\mathfrak{X}|^{\nu_0} |\mathfrak{X}_{11}|^{\nu_1} |\mathfrak{X}_{22}|^{\nu_2}\} \\
&\quad\geq \EE\{\etr(T \mathfrak{X})\} \, \EE(|\mathfrak{X}_T|^{-\nu_0})^{-1} \, \EE\{|(\mathfrak{X}_T)_{11}|^{\nu_1}\} \, \EE\{|(\mathfrak{X}_T)_{22}|^{\nu_2}\} \\
&\qquad\times \big[1 + |I_{p_1} - P_T P_T^{\top}|^{\alpha_0/2 + \nu_1 + \nu_2} \, \{{}_0F_1(\alpha_0/2; \nu_1 \nu_2 P_T P_T^{\top}) - 1\}\big] \\
&\quad\geq \EE\{\etr(T \mathfrak{X})\} \, \EE(|\mathfrak{X}_T|^{-\nu_0})^{-1} \, \EE\{|(\mathfrak{X}_T)_{11}|^{\nu_1}\} \, \EE\{|(\mathfrak{X}_T)_{22}|^{\nu_2}\}.
\end{aligned}
\]
In particular, if $T = 0_{p\times p}$ and $\nu_0 = 0$, the above reduces to
\[
\begin{aligned}
\EE(|\mathfrak{X}_{11}|^{\nu_1} |\mathfrak{X}_{22}|^{\nu_2})
&\geq \EE(|\mathfrak{X}_{11}|^{\nu_1}) \, \EE(|\mathfrak{X}_{22}|^{\nu_2}) \\
&\quad\times \!\big[1 + |I_{p_1} - P P^{\top}|^{\alpha/2 + \nu_1 + \nu_2} \, \{{}_0F_1(\alpha/2; \nu_1 \nu_2 P P^{\top}) - 1\}\big] \\
&\geq \EE(|\mathfrak{X}_{11}|^{\nu_1}) \, \EE(|\mathfrak{X}_{22}|^{\nu_2}).
\end{aligned}
\]
\end{corollary}

\section{Proofs}\label{sec:proofs}

This section contains proofs of Theorem~\ref{thm:1} and Corollary~\ref{cor:1}. Lemmas~\ref{lem:1} and \ref{lem:2}, which are used along the way, are given in Appendix~\ref{app:tech.lemmas}.

\begin{proof}[Proof of Theorem~\ref{thm:1}]
Throughout the proof, for $i\in \{0,1,2\}$, write $\alpha_i = \alpha + 2\nu_i$. To establish the first claim of the theorem, let $\nu_1 > -\alpha/2 + (p_1 - 1)/2$ and $\nu_2 > -\alpha/2 + (p_2 - 1)/2$ be given. This choice ensures that the expectations and integrals are finite everywhere in the proof.

Recall the notation $A = \Sigma^{-1}/2$. Given that
\[
|X| = |X / X_{22}| |X_{22}| = |X_{11} - X_{12} X_{22}^{-1} X_{21}| |X_{22}|,
\]
one has\vspace{1mm}
\[
\begin{aligned}
&\EE(|\mathfrak{X}_{11}|^{\nu_1} |\mathfrak{X}_{22}|^{\nu_2})
= \frac{|A|^{\alpha/2}}{\Gamma_p(\alpha/2)} \int_{\mathcal{S}_{++}^p} |X_{11}|^{\nu_1} |X_{22}|^{\nu_2} \frac{|X|^{\alpha/2 - (p + 1)/2}}{\etr(A X)} \rd X \\
&\quad= \frac{|A|^{\alpha/2}}{\Gamma_p(\alpha/2)} \int_{\mathcal{S}_{++}^p} |X_{11}|^{\nu_1} |X_{22}|^{\alpha_2/2 - (p + 1)/2} \frac{|X_{11} - X_{12} X_{22}^{-1} X_{21}|^{\alpha/2 - (p + 1)/2}}{\etr(A X)} \rd X.
\end{aligned}
\]

Consider the transformation $X_{12} \mapsto X_{11}^{1/2} X_{12} X_{22}^{1/2}$. The Jacobian determinant is $\smash{|X_{11}|^{p_2/2} |X_{22}|^{p_1/2}}$; see, e.g., Theorem~2.1.5 of \citet{MR652932}. Under this transformation, the Schur complement $X / X_{22}$ becomes
\[
\begin{aligned}
X / X_{22}
&\mapsto X_{11} - X_{11}^{1/2} X_{12} X_{22}^{1/2} X_{22}^{-1} X_{22}^{1/2} X_{21} X_{11}^{1/2} \\
&= X_{11}^{1/2} (I_{p_1} - X_{12} X_{21}) X_{11}^{1/2},
\end{aligned}
\]
and the exponential trace factor, $\etr(A X)$, becomes
\[
\etr(A X) \mapsto \etr(A_{11} X_{11}) \etr(A_{22} X_{22}) \etr(2 A_{12} X_{22}^{1/2} X_{21} X_{11}^{1/2}).
\]

Therefore, $\EE (|\mathfrak{X}_{11}|^{\nu_1} |\mathfrak{X}_{22}|^{\nu_2})$ is equal to
\begin{multline*}
\frac{|A|^{\alpha/2}}{\Gamma_p(\alpha/2)} \int_{\mathcal{S}_{++}^{p_1}} \frac{|X_{11}|^{\alpha_1/2 - (p_1 + 1)/2}}{\etr(A_{11} X_{11})} \int_{\mathcal{S}_{++}^{p_2}} \frac{|X_{22}|^{\alpha_2/2 - (p_2 + 1)/2}}{\etr(A_{22} X_{22})} \\
\times \int_{X_{12} X_{21}\in \mathcal{D}_{p_1}} \frac{|I_{p_1} - X_{12} X_{21}|^{\alpha/2 - (p + 1)/2}}{\etr(2 X_{11}^{1/2} A_{12} X_{22}^{1/2} X_{21})} \rd X_{12} \rd X_{22} \rd X_{11},
\end{multline*}
where $\mathcal{D}_{p_1} = \{M\in \R^{p_1\times p_1} : M\in \mathcal{S}_{++}^{p_1} ~~\text{and}~~ I_{p_1} - M\in \mathcal{S}_{++}^{p_1}\}$.

Applying Eq.~(151) of \citet{MR181057} with $b = \alpha/2$, $Y = X_{12}$, $\smash{X = -2 X_{11}^{1/2} A_{12} X_{22}^{1/2}}$, and using the fact that $p_1 \leq p_2$ by assumption, one finds that the innermost integral above is equal to
\[
\frac{\pi^{p_1 p_2/2} \Gamma_{p_1}(\alpha/2 - p_2/2)}{\Gamma_{p_1}(\alpha/2)} \times {}_0F_1(\alpha/2; A_{21} X_{11} A_{12} X_{22}),
\]
upon invoking the invariance under cyclic permutations of the generalized hypergeometric function, ${}_0F_1$, in its matrix argument. By applying the expression for the multivariate gamma function in \eqref{eq:Gamma.m.prod}, one obtains
\[
\begin{aligned}
\Gamma_p(\alpha/2)
&= \frac{\pi^{p(p - 1)/4}}{\pi^{p_1(p_1 - 1)/4} \pi^{p_2(p_2 - 1)/4}} \Gamma_{p_2}(\alpha/2) \Gamma_{p_1}(\alpha/2 - p_2/2) \\[1.5mm]
&= \pi^{p_1 p_2/2} \Gamma_{p_2}(\alpha/2) \Gamma_{p_1}(\alpha/2 - p_2/2).
\end{aligned}
\]
Hence, one can write\vspace{-2mm}
\[
\begin{aligned}
\EE(|\mathfrak{X}_{11}|^{\nu_1} |\mathfrak{X}_{22}|^{\nu_2})
&= \frac{|A|^{\alpha/2}}{\Gamma_{p_1}(\alpha/2) \Gamma_{p_2}(\alpha/2)} \int_{\mathcal{S}_{++}^{p_1}} \frac{|X_{11}|^{\alpha_1/2 - (p_1 + 1)/2}}{\etr(A_{11} X_{11})} \\
&\quad\times\int_{\mathcal{S}_{++}^{p_2}} \frac{|X_{22}|^{\alpha_2/2 - (p_2 + 1)/2} {}_0F_1(\alpha/2; A_{21} X_{11} A_{12} X_{22})}{\etr(A_{22} X_{22})} \rd X_{22} \rd X_{11}.
\end{aligned}
\]

Now, by applying the transformation $X_{22}\mapsto A_{22}^{-1/2} X_{22} A_{22}^{-1/2}$, which has Jacobian determinant $\smash{|A_{22}|^{-(p_2 + 1)/2}}$ by Theorem~2.1.6 of \citet{MR652932}, followed by the recurrence relation for the generalized hypergeometric functions of matrix argument in Eq.~(2.1$'$) of \citet{MR69960}, one has
\[
\begin{aligned}
&\int_{\mathcal{S}_{++}^{p_2}} \frac{|X_{22}|^{\alpha_2/2 - (p_2 + 1)/2} {}_0F_1(\alpha/2; A_{21} X_{11} A_{12} X_{22})}{\etr(A_{22} X_{22})} \rd X_{22} \\[-0.5mm]
&\quad= \int_{\mathcal{S}_{++}^{p_2}} \frac{|X_{22}|^{\alpha_2/2 - (p_2 + 1)/2} {}_0F_1(\alpha/2; A_{22}^{-1/2} A_{21} X_{11} A_{12} A_{22}^{-1/2} X_{22})}{|A_{22}|^{\alpha_2/2} \etr(X_{22})} \rd X_{22} \\
&\quad= \frac{\Gamma_{p_2}(\alpha_2/2)}{|A_{22}|^{\alpha_2/2}} \times {}_1F_1(\alpha_2/2; \alpha/2; A_{12} A_{22}^{-1} A_{21} X_{11}).
\end{aligned}
\]

By applying the transformation $\smash{X_{11}\mapsto A_{11}^{-1/2} X_{11} A_{11}^{-1/2}}$, which has Jacobian determinant $\smash{|A_{11}|^{-(p_1 + 1)/2}}$ by Theorem~2.1.6 of \citet{MR652932}, followed again by the recurrence relation for the generalized hypergeometric functions of matrix argument \cite[Eq.~(2.1$'$)]{MR69960}, one finds
\[
\begin{aligned}
&\int_{\mathcal{S}_{++}^{p_1}} \frac{|X_{11}|^{\alpha_1/2 - (p_1 + 1)/2} {}_1F_1(\alpha_2/2; \alpha/2; A_{12} A_{22}^{-1} A_{21} X_{11})}{\etr(A_{11} X_{11})} \rd X_{11} \\[-0.5mm]
&\quad= \int_{\mathcal{S}_{++}^{p_1}} \frac{|X_{11}|^{\alpha_1/2 - (p_1 + 1)/2} {}_1F_1(\alpha_2/2; \alpha/2; A_{11}^{-1/2} A_{12} A_{22}^{-1} A_{21} A_{11}^{-1/2} X_{11})}{|A_{11}|^{\alpha_1/2} \etr(X_{11})} \rd X_{11} \\
&\quad= \frac{\Gamma_{p_1}(\alpha_1/2)}{|A_{11}|^{\alpha_1/2}} \times {}_2F_1(\alpha_1/2, \alpha_2/2; \alpha/2; A_{11}^{-1} A_{12} A_{22}^{-1} A_{21}).
\end{aligned}
\]

Collecting together the last three equations, and noting that $A_{11}^{-1} A_{12} A_{22}^{-1} A_{21}$ is a cyclic permutation of $P P^{\top}$ if $\smash{P = A_{11}^{-1/2} A_{12} A_{22}^{-1/2}}$, one deduces that
\[
\begin{aligned}
\EE(|\mathfrak{X}_{11}|^{\nu_1} |\mathfrak{X}_{22}|^{\nu_2})
&= \frac{|A|^{\alpha/2}}{\Gamma_{p_1}(\alpha/2) \Gamma_{p_2}(\alpha/2)}
\times \frac{\Gamma_{p_1}(\alpha_1/2)}{|A_{11}|^{\alpha_1/2}}
\times \frac{\Gamma_{p_2}(\alpha_2/2)}{|A_{22}|^{\alpha_2/2}} \\[1mm]
&\quad\times {}_2F_1(\alpha_1/2, \alpha_2/2; \alpha/2; P P^{\top}).
\end{aligned}
\]
Note that
\[
\frac{|A|^{\alpha/2}}{|A_{11}|^{\alpha_1/2} |A_{22}|^{\alpha_2/2}}
= \left(\frac{|A_{11}|}{|A|}\right)^{\nu_2} \left(\frac{|A_{22}|}{|A|}\right)^{\nu_1} \left(\frac{|A|}{|A_{11}| |A_{22}|}\right)^{\alpha/2 + \nu_1 + \nu_2}.
\]

Using the formulas for the inverse of a $2\times 2$ block matrix in Lemma~\ref{lem:1}, one has
\[
\frac{|A_{11}|}{|A|}
= \frac{|(\Sigma^{-1})_{11}/2|}{|\Sigma^{-1}/2|} = \frac{|(\Sigma/\Sigma_{22})^{-1}/2|}{|\Sigma^{-1}/2|} = \frac{|2 \Sigma|}{|2 (\Sigma/\Sigma_{22})|} = |2 \Sigma_{22}|,
\]
\[
\frac{|A_{22}|}{|A|}
= \frac{|(\Sigma^{-1})_{22}/2|}{|\Sigma^{-1}/2|} = \frac{|(\Sigma/\Sigma_{11})^{-1}/2|}{|\Sigma^{-1}/2|} = \frac{|2 \Sigma|}{|2 (\Sigma/\Sigma_{11})|} = |2 \Sigma_{11}|,
\]
and
\[
\frac{|A|}{|A_{11}| |A_{22}|}
= \frac{|A/A_{22}|}{|A_{11}|} = \frac{|A_{11}|^{1/2} |I_{p_1} - P P^{\top}| |A_{11}|^{1/2}}{|A_{11}|} = |I_{p_1} - P P^{\top}|.
\]
Therefore, the latter expectation can be rewritten as
\[
\begin{aligned}
\EE(|\mathfrak{X}_{11}|^{\nu_1} |\mathfrak{X}_{22}|^{\nu_2})
&= |2 \Sigma_{11}|^{\nu_1} \frac{\Gamma_{p_1}(\alpha_1/2)}{\Gamma_{p_1}(\alpha/2)} \, |2 \Sigma_{22}|^{\nu_2} \frac{\Gamma_{p_2}(\alpha_2/2)}{\Gamma_{p_2}(\alpha/2)} \\
&\quad\times |I_{p_1} - P P^{\top}|^{\alpha/2 + \nu_1 + \nu_2} {}_2F_1(\alpha_1/2, \alpha_2/2; \alpha/2; P P^{\top}).
\end{aligned}
\]

Moreover, by Herz's generalization of Euler's transformation for ${}_2F_1$, the Gaussian hypergeometric function of matrix argument (\cite[Eq.~(6.3)]{MR69960}, \cite[Eq.~(35.7.6)]{MR2723248}), one has
\[
|I_{p_1} - P P^{\top}|^{\alpha/2 + \nu_1 + \nu_2} {}_2F_1(\alpha_1/2, \alpha_2/2; \alpha/2; P P^{\top})
= {}_2F_1(-\nu_1, -\nu_2; \alpha/2; P P^{\top}).
\]
Hence, Lemma~\ref{lem:2} implies
\[
\EE(|\mathfrak{X}_{11}|^{\nu_1} |\mathfrak{X}_{22}|^{\nu_2})
= \EE(|\mathfrak{X}_{11}|^{\nu_1}) \, \EE(|\mathfrak{X}_{22}|^{\nu_2}) \, {}_2F_1(-\nu_1, -\nu_2; \alpha/2; P P^{\top}),
\]
which proves \eqref{thm:1.eq.1}.

To establish the second claim of the theorem, recall that $\alpha_0 = \alpha + 2\nu_0$, and let $\nu_0 > -\alpha/2 + (p - 1)/2$, $\nu_1 > -\alpha_0/2 + (p_1 - 1)/2$, $\nu_2 > -\alpha_0/2 + (p_2 - 1)/2$ be given. Again, this choice ensures that the expectations and integrals are finite everywhere. Let $T$ be a $p\times p$ real symmetric matrix such that $\Sigma^{-1} - 2T\in \mathcal{S}_{++}^p$, and consider
\[
\mathfrak{X}_T \sim \mathcal{W}_p\{\alpha_0, (\Sigma^{-1} - 2T)^{-1}\}.
\]
A straightforward renormalization argument yields
\[
\begin{aligned}
&\EE\{\etr(T \mathfrak{X}) \, |\mathfrak{X}|^{\nu_0} |\mathfrak{X}_{11}|^{\nu_1} |\mathfrak{X}_{22}|^{\nu_2}\} \\
&\quad= \frac{|2 (\Sigma^{-1} - 2T)^{-1}|^{\alpha_0/2} \Gamma_p(\alpha_0/2)}{|2 \Sigma|^{\alpha/2} \Gamma_p(\alpha/2)} \, \EE\{|(\mathfrak{X}_T)_{11}|^{\nu_1} |(\mathfrak{X}_T)_{22}|^{\nu_2}\} \\
&\quad= |I_p - 2 T \Sigma|^{-\alpha/2} \, |2 (\Sigma^{-1} - 2 T)^{-1}|^{\nu_0} \, \frac{\Gamma_p(\alpha_0/2)}{\Gamma_p(\alpha/2)} \, \EE\{|(\mathfrak{X}_T)_{11}|^{\nu_1} |(\mathfrak{X}_T)_{22}|^{\nu_2}\}.
\end{aligned}
\]
By applying \eqref{thm:1.eq.1}, Lemma~\ref{lem:2} and the fact that $\EE\{\etr(T \mathfrak{X})\} = |I_p - 2 T \Sigma|^{-\alpha/2}$ by~\eqref{eq:Lt}, one deduces that
\[
\begin{aligned}
&\EE\{\etr(T \mathfrak{X}) \, |\mathfrak{X}|^{\nu_0} |\mathfrak{X}_{11}|^{\nu_1} |\mathfrak{X}_{22}|^{\nu_2}\} \\
&\quad= \EE\{\etr(T \mathfrak{X})\} \, \EE(|\mathfrak{X}_T|^{-\nu_0})^{-1} \, \EE\{|(\mathfrak{X}_T)_{11}|^{\nu_1}\} \, \EE\{|(\mathfrak{X}_T)_{22}|^{\nu_2}\} \\
&\qquad\times {}_2F_1(-\nu_1, -\nu_2; \alpha_0/2; P_T P_T^{\top}).
\end{aligned}
\]
This proves \eqref{thm:1.eq.2} and concludes the argument.
\end{proof}

\begin{proof}[Proof of Corollary~\ref{cor:1}]
Let $m\in \N$ be given and define the set
\[
\mathcal{D}_m = \{M\in \R^{m\times m} : M\in \mathcal{S}_{++}^m ~~\text{and}~~ I_m - M\in \mathcal{S}_{++}^m\}.
\]
To establish the claim of the corollary, it is sufficient to show that for every $a, b\in [0, \infty)$, $c > (m - 1)/2$, and $M \in \mathcal{D}_m$, one has
\begin{equation}\label{eq:lower.bound.2F1}
{}_2F_1(-a,-b;c;M) \geq 1 + |I_m - M|^{c+a+b} \, \{{}_0F_1(c;abM) - 1\} \geq 1.
\end{equation}

If $a = 0$ or $b = 0$, then \eqref{eq:lower.bound.2F1} is trivial. Therefore, assume for the remainder of the proof that $a,b\in (0,\infty)$.

Consider the following identity, which holds for all $x,y,z\in \R$,
\[
(z + x) (z + y) = z (z + x + y) + xy.
\]
By applying this identity to the rising factorials defined in Section~\ref{sec:definitions}, one obtains, for any $\tilde{c}\in (0,\infty)$ and any $k\in \N$,
\[
\begin{aligned}
(\tilde{c}+a)_k (\tilde{c}+b)_k
&= \prod_{j=1}^k \{(\tilde{c}+j-1)+a\} \{(\tilde{c}+j-1)+b\} \\
&= \prod_{j=1}^k \{(\tilde{c}+j-1)(\tilde{c}+j-1+a+b) + ab\} \\
& \geq \prod_{j=1}^k (\tilde{c}+j-1)(\tilde{c}+j-1+a+b) + (ab)^k \\
&= (\tilde{c})_k (\tilde{c}+a+b)_k + (ab)^k.
\end{aligned}
\]
Next, by applying the latter inequality to the partitional rising factorials defined in \eqref{eq:partitional.rising}, one sees that, for any $c > (m - 1)/2$, and any partition $\bb{\kappa} = (k_1,\ldots,k_m)$ such that $k_1 \geq \dots \geq k_{m_0} > 0 = k_{m_0+1} = \dots = k_m$ for some integer $m_0 \leq m$, one has
\[
\begin{aligned}
(c+a)_{\bb{\kappa}} \, (c+b)_{\bb{\kappa}}
&= \prod_{j=1}^{m_0} (c-(j-1)/2+a)_{k_j} (c-(j-1)/2+b)_{k_j} \\
& \geq \prod_{j=1}^{m_0} \{(c-(j-1)/2)_{k_j} (c-(j-1)/2+a+b)_{k_j} + (ab)^{k_j}\} \\
& \geq \prod_{j=1}^{m_0} (c-(j-1)/2)_{k_j} (c+a+b-(j-1)/2)_{k_j} + \prod_{j=1}^{m_0} (ab)^{k_j} \\
&= (c)_{\bb{\kappa}} \, (c+a+b)_{\bb{\kappa}} + (ab)^{|\bb{\kappa}|}.
\end{aligned}
\]
As a result, for all $\bb{\kappa} \neq \bb{0}_m$,
\[
\frac{(c+a)_{\bb{\kappa}} \, (c+b)_{\bb{\kappa}}}{(c)_{\bb{\kappa}}} \geq (c+a+b)_{\bb{\kappa}} + \frac{(ab)^{|\bb{\kappa}|}}{(c)_{\bb{\kappa}}}.
\]
Multiplying both sides by $C_{\bb{\kappa}}(M)/|\bb{\kappa}|!$, and summing over all partitions $\bb{\kappa}$, one finds
\[
\begin{aligned}
{}_2F_1(c+a,c+b;c;M)
& \geq {}_1F_0(c+a+b;M) + {}_0F_1(c;abM) - 1 \\
&= |I_m - M|^{-(c+a+b)} + {}_0F_1(c;abM) - 1;
\end{aligned}
\]
see, e.g., \citet[p.~485]{MR69960} for the last equality. It readily follows that
\[
|I_m - M|^{c+a+b} {}_2F_1(c+a,c+b;c;M) \geq 1 + |I_m - M|^{c+a+b} \{{}_0F_1(c;abM) - 1\}.
\]
The left-hand side is equal to ${}_2F_1(-a,-b;c;M)$ by Herz's generalization of Euler's transformation; see \cite[Eq.~(6.3)]{MR69960} or \cite[Eq.~(35.7.6)]{MR2723248}. The right-hand is larger than one because the partitional rising factorials and zonal polynomials are all positive in ${}_0F_1(c;abM) - 1$, given the restrictions on $a,b,c$ and $M$. This concludes the proof of the corollary.
\end{proof}

\appendix

\section{Numerical validation}\label{app:numerical.validation}

Matlab codes and tables which numerically validate the results of Theorem~\ref{thm:1} and Corollary~\ref{cor:1} are available in \cite{GenestOuimetRichards2024github}. Hypergeometric functions of matrix argument are computed using the \texttt{mhg} command introduced by \citet{MR2196994}.

\section{Technical lemmas}\label{app:tech.lemmas}

The first lemma contains well-known formulas from linear algebra for the inverse of a $2\times 2$ block matrix; see, e.g., Theorem~2.1 of \citet{MR1873248}.

\begin{lemma}\label{lem:1}
Let $M$ be a matrix of width at least $2$ which is partitioned into a $2\times 2$ block matrix as follows:
\[
M =
\begin{bmatrix}
M_{11} & M_{12} \\
M_{21} & M_{22}
\end{bmatrix}.
\]
\begin{enumerate}[1.]
\item
If $M_{11}$ is invertible and the Schur complement $M / M_{11} = M_{22} - M_{21} M_{11}^{-1} M_{12}$ is invertible, then
\[
M^{-1} =
\begin{bmatrix}
M_{11}^{-1} + M_{11}^{-1} M_{12} (M / M_{11})^{-1} M_{21} M_{11}^{-1} & - M_{11}^{-1} M_{12} (M / M_{11})^{-1} \\[1mm]
- (M / M_{11})^{-1} M_{21} M_{11}^{-1} & (M / M_{11})^{-1}
\end{bmatrix}.
\]
\item
If $M_{22}$ is invertible and the Schur complement $M / M_{22} = M_{11} - M_{12} M_{22}^{-1} M_{21}$ is invertible, then
\[
M^{-1} =
\begin{bmatrix}
(M / M_{22})^{-1} & - (M / M_{22})^{-1} M_{12} M_{22}^{-1} \\[1mm]
- M_{22}^{-1} M_{21} (M / M_{22})^{-1} & M_{22}^{-1} + M_{22}^{-1} M_{21} (M / M_{22})^{-1} M_{12} M_{22}^{-1}
\end{bmatrix}.
\]
\end{enumerate}
\end{lemma}

The second lemma provides a general expression for the expectation of any principal minor of a Wishart random matrix. It is a consequence of Corollary 3.2.6 and p.~101 of \citet{MR652932}.

\begin{lemma}\label{lem:2}
For any $p_1,p_2\in \N$, let $p = p_1 + p_2$, and fix any degree-of-freedom $\alpha \in (p-1, \infty)$. Let $\Sigma\in \mathcal{S}_{++}^p$ be any $p \times p$ positive definite matrix expressed in the form $(\Sigma_{ij})$, where for each $i,j\in \{1,2\}$, the block $\Sigma_{ij}$ has size~$p_i \times p_j$. Assume that
\[
\mathfrak{X}
= \begin{bmatrix}
\mathfrak{X}_{11} & \mathfrak{X}_{12} \\
\mathfrak{X}_{21} & \mathfrak{X}_{22} \\
\end{bmatrix} \sim \mathcal{W}_p(\alpha,\Sigma),
\quad \text{with}~~~
\Sigma =
\begin{bmatrix}
\Sigma_{11} & \Sigma_{12} \\
\Sigma_{21} & \Sigma_{22} \\
\end{bmatrix}.
\]
Then, for every $i\in \{1,2\}$ and real $\nu_i > -\alpha/2 + (p_i - 1)/2$, one has
\[
\EE(|\mathfrak{X}_{ii}|^{\nu_i}) = |2 \Sigma_{ii}|^{\nu_i} \frac{\Gamma_{p_i}(\alpha/2 + \nu_i)}{\Gamma_{p_i}(\alpha/2)}.
\]
\end{lemma}

\bibliographystyle{plainnat}
\bibliography{PAMS_bib}

\end{document}